\documentclass[reqno]{amsart}
\usepackage[notref,notcite]{showkeys}

\begin{document}
\title[Solutions to Navier-Stokes equation]
{Solutions to three-dimensional Navier-Stokes equations
for incompressible fluids}

\author[J. Jormakka]
{Jorma Jormakka}

\address{Jorma Jormakka \newline
National Defence University, Adjunct Professor, Home address: Karhekuja 4, 01660 Vantaa,
 Finland}
\email{jorma.o.jormakka@kolumbus.fi}


\begin{abstract}
 This article is an updated version of the article that was published 
 in the Electronic Journal of Differential Equations on 10. July 2010. 
 Two footnotes have been added. One corrects a minor error not influencing 
 the proof, the second is only a clarifying text to the existing proof.  
 A discussion how the published article solves the Clay Millennium Prize 
 problem on 
 the Navier-Stokes equations is added, the critizism against the 
 published article is answered and a discussion how the Clay problem 
 statement should be corrected is included.
  
 The article gives explicit solutions to the space-periodic
 Navier-Stokes problem with non-periodic pressure.
 These type of solutions are not unique and by using such solutions
 one can construct a periodic, smooth, divergence-free initial
 vector field allowing a space-periodic and time-bounded external
 force such that there exists a smooth solution to the
 3-dimensional Navier-Stokes equations for incompressible fluid
 with those initial conditions, but the solution cannot be continued
 to the whole space.
\end{abstract}

\maketitle
\numberwithin{equation}{section}
\newtheorem{theorem}{Theorem}[section]
\newtheorem{lemma}[theorem]{Lemma}
\allowdisplaybreaks

\section{Introduction}

Let $x=(x_1,x_2,x_3)\in \mathbb{R}^3$ denote the position,
$t\ge 0$ be the time,
$p(x,t)\in \mathbb{R}$ be the pressure and
$u(x,t)=(u_i(x,t))_{1\le i\le 3}\in \mathbb{R}^3$ be the velocity vector.
Let $f_i(x,t)$ be the external force.
The Navier-Stokes equations for incompressible
fluids filling all of $\mathbb{R}^3$ for $t\ge 0$ are \cite{feff}
\begin{gather}
\frac{\partial u_i}{\partial t}+\sum_{j=1}^3
{u_j\frac{\partial u_i}{\partial x_j}}
=  v \Delta u_i -\frac{\partial p}{\partial x_i}+f_i(x,t),
\quad x\in\mathbb{R}^3,\; t\ge 0,\; 1\le i\le 3 \label{e1}\\
\operatorname{div}u=\sum_{i=1}^3 \frac{\partial u_i}{\partial x_i}=0,
\quad x\in\mathbb{R}^3,\; t\ge 0 \label{e2}
\end{gather}
with initial conditions
\begin{equation}
u(x,0)=u^0(x),\quad x\in \mathbb{R}^3.\label{e3}
\end{equation}
Here $\Delta=\sum_{i=1}^3 \frac{\partial^2}{\partial x_i^2}$
is the Laplacian in the
space variables, $ v $ is a positive coefficient and $u^0(x)$ is
$C^{\infty}(\mathbb{R}^3)$
vector field on $\mathbb{R}^3$ required to be divergence-free;
 i.e., satisfying $\operatorname{div}u^0=0$. The time derivative
$\frac{\partial u_i }{\partial t}$
at $t=0$ in \eqref{e1} is taken to mean the limit when $t\to 0^+$.

This article shows that there exists
 $C^{\infty}(\mathbb{R}^3)$, periodic, divergence-free initial
vector fields $u^0$ defined at $\mathbb{R}^3$ such that there exists
a family of smooth (here, in the class
$C^{\infty}(\mathbb{R}^3\times [0,\infty))$)
functions $u(x,t)$ and $p(x,t)$ satisfying \eqref{e1}, \eqref{e2}
and \eqref{e3}.
We also show that there exist a periodic and bounded external
force $f_i(x,t)$ such that the solution cannot be continued to
the whole $\mathbb{R}^3\times [0,\infty)$.

\section{Theorems and Lemmas}

The simple explicit case of $u(x,t)$ in Lemma \ref{lem1} satisfies the
conditions given in \eqref{e3} and allows a free function $g(t)$
that only satisfies $g(0)=g'(0)=0$. The solution is
then not unique. If the time derivatives of $u(x,t)$ are specified at
$t=0$ then the solution in the lemma is unique.
Footnote added: This is where Grigori Rozenblioum noticed a small error:  
the statement that solutions are unique requires $g(t)$ to be analytic. 
Lemma 2.1 only requires $g(t)$ to be smooth. This comment of Rozenblioum is 
true, but the sentence is not in a theorem, nor is it used in the paper. It
is in an informal explanation of Lemma 2.1 and should be understood in the 
context of theorems of strong solutions being unique when initial conditions
are fully specified. In that context, the solutions are assumed analytic.

\begin{lemma} \label{lem1}
Let
\begin{gather*}
u_1^0=2\pi \sin (2\pi x_2)+2\pi \cos(2\pi x_3),\\
u_2^0=2\pi \sin (2\pi x_3)+2\pi \cos(2\pi x_1),\\
u_3^0=2\pi \sin (2\pi x_1)+2\pi \cos(2\pi x_2)
\end{gather*}
be the initial vector field, and let $f_i(x,t)$ be chosen identically
zero for $1\le i \le3$.
Let $g:\mathbb{R}\to \mathbb{R}$ be a smooth function with
$g(0)=g'(0)=0$ and $\beta=(2\pi)^2  v $.
The following family of functions $u$ and $p$ satisfy
\eqref{e1}-\eqref{e3}:
\begin{equation} \label{e4}
\begin{gathered}
u_1=e^{-\beta t}2\pi
\left(\sin (2\pi (x_2+g(t)))+\cos (2\pi (x_3+g(t)))\right)-g'(t),\\
u_2=e^{-\beta t}2\pi \left(\sin (2\pi (x_3+g(t)))
 +\cos (2\pi (x_1+g(t)))\right) -g'(t),\\
u_3=e^{-\beta t}2\pi \left(\sin (2\pi (x_1+g(t)))
 +\cos (2\pi (x_2+g(t)))\right)-g'(t),\\
\begin{aligned}
p&=  -e^{-2\beta t}(2\pi)^2 \sin (2\pi (x_1+g(t)))
 \cos (2\pi (x_2+g(t)))\\
&\quad -e^{-2\beta t}(2\pi)^2 \sin (2\pi (x_2+g(t)))
 \cos (2\pi (x_3+g(t)))\\
&\quad -e^{-2\beta t}(2\pi)^2 \sin (2\pi (x_3+g(t)))
 \cos (2\pi (x_1+g(t)))  +g''(t)\sum_{j=1}^3 x_j.
\end{aligned}
\end{gathered}
\end{equation}
\end{lemma}

\begin{proof}
The initial vector field is smooth, periodic, bounded and
divergence-free.
Let $(i,k,m)$ be any of the permutations $(1,2,3)$, $(2,3,1)$ or
$(3,1,2)$.
We can write all definitions in \eqref{e4} shorter as
(here $g'(t)=dg/dt$):
\begin{gather*}
u_i=e^{-\beta t}2\pi \left( \sin (2\pi (x_k+g(t)))
 +\cos (2\pi (x_m+g(t)))\right)-g'(t),\\
\begin{aligned}
p&=-e^{-2\beta t}(2\pi)^2 \sin (2\pi (x_i+g(t)))
 \cos (2\pi (x_k+g(t)))\\
&\quad   -e^{-2\beta t}(2\pi)^2 \sin (2\pi (x_k+g(t)))
 \cos (2\pi (x_m+g(t)))\\
&\quad  -e^{-2\beta t}(2\pi)^2 \sin (2\pi (x_m+g(t)))
\cos (2\pi (x_i+g(t)))
   +g''(t)\sum_{j=1}^3 x_j.
\end{aligned}
\end{gather*}
It is sufficient to proof the claim for these permutations. The
permutations $(1,3,2)$, $(2,1,3)$ and $(3,2,1)$ only interchange
the indices $k$ and $m$. The functions \eqref{e4} are smooth
and $u(x,t)$ in \eqref{e4}
satisfies \eqref{e2} and \eqref{e3} for the initial vector
field in Lemma \ref{lem1}.
We will verify \eqref{e1} by directly computing:
\begin{gather*}
\begin{aligned}
\frac{\partial u_i}{\partial t}
&=-\beta e^{-\beta t}2\pi \left(
\sin(2\pi (x_k+g(t)))+\cos(2\pi (x_m+g(t)))\right) -g''(t) \\
&\quad +g'(t)e^{-\beta t}(2\pi)^2
\left(\cos(2\pi (x_k+g(t)))-\sin(2\pi (x_m+g(t)))\right),
\end{aligned} \\
- v \Delta u_i
= v e^{-\beta t}(2\pi)^3\left(
\sin (2\pi (x_k+g(t)))+\cos (2\pi (x_m+g(t)))\right),\\
\begin{aligned}
\frac{\partial p}{\partial x_i}
&=  -e^{-2\beta t}(2\pi)^3 \cos (2\pi (x_i+g(t)))
 \cos (2\pi (x_k+g(t)))\\
&\quad +e^{-2\beta t}(2\pi)^3 \sin (2\pi (x_m+g(t)))
 \sin (2\pi (x_i+g(t)))+g''(t).
\end{aligned}
\end{gather*}
The functions $u_k$ and $u_m$ are
\begin{gather*}
u_k=e^{-\beta t}2\pi \left( \sin (2\pi (x_m+g(t)))
 +\cos (2\pi (x_i+g(t)))\right)-g'(t),\\
 u_m=e^{-\beta t}2\pi \left( \sin (2\pi (x_i+g(t)))
 +\cos (2\pi (x_k+g(t)))\right)-g'(t).
\end{gather*}
The remaining term to be computed in \eqref{e1} is
\begin{align*}
\sum_{j\in \{i,k,m\}}u_j\frac{\partial u_i}{\partial x_j}
&= u_i\frac{\partial u_i}{\partial x_i}
 + u_k\frac{\partial u_i}{\partial x_k}
+u_m\frac{\partial u_i}{\partial x_m}\\
&=e^{-2\beta t}(2\pi)^3\cos (2\pi (x_i+g(t)))\cos (2\pi (x_k+g(t)))\\
&\quad -e^{-2\beta t}(2\pi)^3\sin (2\pi (x_m+g(t)))
 \sin (2\pi (x_i+g(t)))\\
&\quad -g'(t) e^{-\beta t}(2\pi)^2 \left(
\cos (2\pi (x_k+g(t)))-\sin (2\pi (x_m+g(t)))\right).
\end{align*}
Inserting the parts to \eqref{e1} shows that
$$
\frac{\partial u_i}{\partial t}
+\sum_{j=1}^3 {u_j\frac{\partial u_i}{\partial x_j}}
- v \Delta u_i +\frac{\partial p}{\partial x_i}=0.
$$
This completes the proof .
\end{proof}

\begin{theorem} \label{thm1}
 There exists a periodic,
$C^{\infty}(\mathbb{R}^3)$, and divergence-free
vector field $u^0(x)=(u_i^0(x))_{1\le i\le3}$
on $\mathbb{R}^3$ such that the following two claims hold:
\begin{itemize}
\item[C1:] The solution to  \eqref{e1}-\eqref{e3}
is not necessarily unique. In fact,
there are infinitely many
$C^{\infty}\left(\mathbb{R}^3\times [0,\infty)\right)$
functions $u(x,t)=(u_i(x,t))_{1\le i\le3}$ and
$p(x,t)$ satisfying \eqref{e1}, \eqref{e2} and \eqref{e3}.

\item[C2:] Periodic initial values do not guarantee that the solution
is bounded. Indeed, there exist unbounded
$u(x,t)$ and $p(x,t)$ satisfying \eqref{e1}, \eqref{e2} and the initial
values \eqref{e3}. There also exist bounded solutions that are
periodic as functions of $x$.
\end{itemize}
\end{theorem}

\begin{proof}
Let $f_i(x,t)$ be chosen identically zero for $1\le i \le3$,
and let us select $g(t)={1\over 2}ct^2$ in Lemma \ref{lem1}.
The value $c\in \mathbb{R}$ can be freely chosen. This shows C1.
If $c=0$ then the solution is bounded and it is periodic as a function
of $x$. If $c\not=0$ then $u_i(x,t)$ for every $i$ and $p(x,t)$ are all
unbounded. In $u_i(x,t)$ the $ct=g'(t)$ term and in
$p$ the term $c(x_1+x_2+x_3)=g''(t)\sum_{j=1}^3x_j$
are not bounded. This shows C2.
The failure of uniqueness is caused by the fact that \eqref{e1}-\eqref{e3} do not
determine the limits of the higher time derivatives of $u(x,t)$
when $t\to 0+$.
These derivatives
can be computed by differentiating  \eqref{e4} but the
function $g(t)$ is needed and it determines the higher
time derivatives. As $g(t)$ can be freely
chosen, the solutions are not unique.
\end{proof}

\begin{theorem} \label{thm2}
There exists a smooth, divergence-free vector
field $u^0(x)$ on $\mathbb{R}^3$ and a smooth $f(x,t)$ on
$\mathbb{R}^3\times [0,\infty)$
and a number $C_{\alpha,m,K}>0$ satisfying
\begin{equation}
u^0(x+e_j)=u^0(x), \quad
f(x+e_j,t)=f(x,t), \quad  1\le j\le 3 \label{e5}
\end{equation}
(here $e_j$ is the unit vector),
and
\begin{equation}
|\partial_x^{\alpha}\partial_t^{m}f(x,t)|
\le C_{\alpha m K}(1+|t|)^{-K}\label{e6}
\end{equation}
for any $\alpha$, $m$ and $K$,
such that there exists $a>0$ and a solution $(p,u)$
of \eqref{e1}, \eqref{e2}, \eqref{e3}
satisfying
\begin{equation}
u(x,t)=u(x+e_j,t)\label{e7}
\end{equation}
on $\mathbb{R}^3\times [0,a)$ for $1\le j\le 3$, and
\begin{equation}
p,u\in C^{\infty}(\mathbb{R}^3\times [0,a))\label{e8}
\end{equation}
that cannot be smoothly continued to $\mathbb{R}^3\times [0,\infty)$.
\end{theorem}

\begin{proof}
Let us make a small modification to the solution in
Lemma \ref{lem1}.
In Lemma \ref{lem1}, $g:\mathbb{R}\to \mathbb{R}$ is a smooth function with
$g(0)=g'(0)=0$, but we select
$$
g(t)={1\over 2}ct^2{1\over a-t}
$$
where $c\not=0$ and $a>0$.

The initial vector field $u^0(x)$ in Lemma \ref{lem1} is smooth,
periodic and divergence-free. The period is scaled to one
in \eqref{e4}.
The $f(x,t)$ is zero and therefore is periodic in space variables
with the period as one. Thus, \eqref{e5} holds.
The constant $C_{\alpha,m,K}$ is selected after
the numbers $\alpha,m,K$ are selected.
The force $f(x,t)$ is identically zero, thus \eqref{e6} holds.
The solution \eqref{e4} in Lemma \ref{lem1} is periodic in space
variables with the period as one. Thus \eqref{e7} holds.
The solution $u(x,t)$ in \eqref{e4} is smooth if $t<a$, thus \eqref{e8} holds.
The function $g'(t)$ has a singularity at a finite value $t=a$
and $g'(t)$ becomes infinite at $t=a$.
 From \eqref{e4} it follows that if $t$ approaches
$a$ from either side, there is no limit to the the
oscillating sine and cosine term in $u_1$,
and the $g'(t)$ additive term approaches infinity.
Thus, the solution $u(x,t)$ cannot be continued to the whole
$\mathbb{R}^3\times [0,\infty)$.
This completes the proof.
\end{proof}

\begin{theorem} \label{thm3}
There exists a smooth, divergence-free vector
field $u^0(x)$ on $\mathbb{R}^3$ and a smooth $f(x,t)$ on
$\mathbb{R}^3\times [0,\infty)$
defined as a feedback control function using the values of $u(x,t)$
and a number $C_{\alpha,m,K}>0$ satisfying
\begin{equation}
u^0(x+e_j)=u^0(x) , \quad f(x+e_j,t)=f(x,t) , \quad
1\le j\le 3.\label{e5'}
\end{equation}
(here $e_j$ is the unit vector), and
\begin{equation}
|\partial_x^{\alpha}\partial_t^{m}f(x,t)|
\le C_{\alpha m K}(1+|t|)^{-K} \label{e6'}
\end{equation}
for any $\alpha$, $m$ and $K$,
such that there exist no solutions $(p,u)$
of \eqref{e1}, \eqref{e2}, \eqref{e3} on
$\mathbb{R}^3\times [0,\infty)$ satisfying
\begin{equation}
u(x,t)=u(x+e_j,t)\label{e7'}
\end{equation}
on $\mathbb{R}^3\times [0,\infty)$ for $1\le j\le 3$, and
\begin{equation}
p,u\in C^{\infty}(\mathbb{R}^3\times [0,\infty)).\label{e8'}
\end{equation}
\end{theorem}

\begin{proof}
Let the solution in Theorem \ref{thm2} with the particular $g$ be denoted
by $U$ and $a$ be larger than $1$.
A feedback control force $f(x,t)$ is defined by using the values of
the function $u(x,t')$ for $t'\le t$. In practise there is a control
delay and $t'<t$ but we allow zero control delay and
select $f(x,t)$ as
$$
f_i(x,t)= \frac{\partial}{\partial t}u_i(x,t)
-\frac{\partial}{\partial t}U_i(x,t).
$$
Inserting this force to \eqref{e1} yields a differential
equation in space variables
$$
\frac{\partial U_i}{\partial t}
+\sum_{j=1}^3 {u_j\frac{\partial u_i}{\partial x_j}}
= v \Delta u_i -\frac{\partial p}{\partial x_i}
\quad x\in\mathbb{R}^3,\; t\ge 0, \;1\le i\le 3.
$$
This force is defined in the open interval $t\in [0,a)$
and can be smoothly continued as zero to $[a,\infty)$.
There is a solution $u(x,t)=U(x,t)$ to this equation.
We notice that if $u(x,t)=U(x,t)$ then the force $f_i(x,t)$
takes zero value at every point.
This is correct: if we apply external control force without
any control delay,
it is possible to keep the solution $u$ exactly at the selected solution
$U$. This is not the same as to say that there is no force. If the
solution $u(x,t)$ would be different from $U(x,t)$, then the force
would not be zero.
Since $u_1(x,t)=U_1(x,t)$ becomes infinite when $t$ approaches $a$,
the solution cannot be continued
to the whole space $\mathbb{R}^3\times [0,\infty)$.
As in Theorem \ref{thm2} the conditions \eqref{e5'}-\eqref{e7'} hold,
 but \eqref{e8'} does not hold.

We must still discuss if the feedback control force can control the
equation \eqref{e1} or if there can be several solutions.
When the higher order time derivatives of $u(x,t)$ are fixed at $t=0$
the solution to \eqref{e1} is unique because of the local-in-time existence
and uniqueness theorem. This means that if a solution $u(x,t)$ starts
as $U(x,t)$ in some small interval $t<\epsilon$ for some small $\epsilon<1$,
then it will continue as $U(x,t)$ for all times $t<a$ if the external
force is zero.
The question is whether the feedback control force $f(x,t)$ can
steer the solution $u(x,t)$ to a possible solution $U(x,t)$.
The external force can freely change the higher order time derivatives of
$u(x,t)$ in the equation \eqref{e1}. Thus, the external control force can set
the higher order time derivatives of $u(x,t)$ to those of $U(x,t)$, therefore
the answer is that the external control force can control the equation and the
solution can be set to $U(x,t)$.

Footnote added: This is where Grigori Rozenblioum thinks that the proof 
uses the incorrect claim that if the time derivatives of all orders equal, 
then the functions equal. This is true only for analytic functions. However, 
Rozenblioum does not understand this part of the proof. No such faulty 
claim is used.
The text in [4] says only that $f$ can select $U(x,t)$ and that when it 
selects $U(x,t)$, the solution stays unique because of theorems guaranteeing 
uniqueness of strong solutions under suitable conditions, actually because 
there are no brancing points in the equations. 
In order to explain what the text fully means, 
let us firstly notice that the feedback control force is said to have zero
value. Thus, it cannot steer any solutions to $U(x,t)$. It sets the solution
to have the time derivatives of $U(x,t)$. This means that the feedback control
force acts by selection, not by steering. It is not stated that this selection
gives nonanalytic functions $u(x,t)$. It is instead stated that the selection
gives the unique function $u(x,t)=U(x,t)$. That means that the feedback 
control force $f$ is defined to control only one function $u(x,t)$, i.e., 
$U(x,t)$. Selecting the force $f$ means selecting the function $U(x,t)$. 
Thus, the solution is unique. In the discussion section this issue is explained
in a simple way with examples. 

The difference between Theorems \ref{thm2} and \ref{thm3}
 is that if the force $f(x,t)$
is zero in Theorem \ref{thm2}, there is a family of solutions corresponding
to different selections of $g(t)$, but if the force $f(x,t)$ is zero
in Theorem \ref{thm3}, then necessarily the solution $u(x,t)$ equals $U(x,t)$ because
otherwise the force is not zero.

Let us mention that we may select a force that does not
take the value zero at all points
e.g. by adding a control delay that has a zero value at $t=0$ and
when $t>t_1$ for some fixed $t_1$ satisfying $0<t_1<a$ and smoothing
the force to $C^\infty$.
At some points $t<t_1$ the control delay is selected as nonzero and
consequently the force is not zero at all points.
This completes the proof.
\end{proof}

Let us continue by partially solving \eqref{e1}-\eqref{e3}.
Firstly, it is good to eliminate $p$ by integrability conditions as in
Lemma \ref{lem2}. We introduce new unknowns $h_{i,k}$.
The relation of $h_{i,k}$ and $u_i$ is given by Lemma \ref{lem3}.

\begin{lemma} \label{lem2}
 Let $u(x,t)$ and $p(x,t)$ be
$C^{\infty}\left(\mathbb{R}^3\times [0,\infty)\right)$
functions satisfying
\eqref{e1} and \eqref{e2} with $f_i(x,t)$ being identically
zero for $1\le i\le 3$,
and let $(i,k,m)$ be a permutation of $(1,2,3)$.  The functions
\begin{equation}
h_{i,k}=\frac{\partial u_i}{\partial x_k}
-\frac{\partial u_k}{\partial x_i}\label{e9}
\end{equation}
satisfy
\begin{equation}
\frac{\partial h_{i,k}}{\partial t}
+\sum_{j=1}^3 u_j\frac{\partial h_{i,k}}{\partial x_j}
- v \Delta h_{i,k}\\
=\frac{\partial u_m}{\partial x_m}h_{i,k}
-\frac{\partial u_m}{\partial x_k}
\frac{\partial u_i}{\partial x_m}
+\frac{\partial u_m}{\partial x_i}
\frac{\partial u_k}{\partial x_m} \label{e10}
\end{equation}
for all $x\in \mathbb{R}^3$ and $t\ge 0$.
\end{lemma}

\begin{proof}
As $p\in C^{\infty}\left(\mathbb{R}^3\times [0,\infty)\right)$,
$$
\frac{\partial}{\partial x_i}\frac{\partial p}{\partial x_k}
=\frac{\partial}{\partial x_k}\frac{\partial p}{\partial x_i}.
$$
Thus, from \eqref{e1} when $f_i(x,t)$ is identically zero
for $1\le i\le 3$,
we obtain
\begin{align*}
&\frac{\partial p}{\partial\partial t}
\frac{\partial u_i}{\partial x_k}
+\sum_{j=1}^3\Big( u_j\frac{\partial^2 u_i}{ \partial x_j\partial x_k}
+\frac{\partial u_j}{ \partial x_k}\frac{\partial u_i}{\partial x_j}\Big)
- v \Delta \frac{\partial u_i}{\partial x_k}\\
&=\frac{\partial p}{\partial\partial t}
\frac{\partial u_k}{\partial x_i}
+\sum_{j=1}^3\Big( u_j{\partial^2 u_k\over \partial x_j\partial x_i}
+\frac{\partial u_j}{ \partial x_i}\frac{\partial u_k}{\partial x_j}\Big)
- v \Delta \frac{\partial u_k}{\partial x_i}.
\end{align*}
This yields
\begin{equation}
\frac{\partial p}{\partial\partial t}h_{i,k}
+\sum_{j=1}^3 u_j\frac{\partial h_{i,k}}{\partial x_j}
- v \Delta h_{i,k}
=\sum_{j=1}^3\frac{\partial u_j}{\partial x_i}
\frac{\partial u_k}{\partial x_j}
-\sum_{j=1}^3\frac{\partial u_j}{\partial x_k}
\frac{\partial u_i}{\partial x_j}.\label{e11}
\end{equation}
The right-hand side of \eqref{e11} can be written in the form
\begin{equation}
\begin{aligned}
&\frac{\partial u_i}{\partial x_i}\frac{\partial u_k}{\partial x_i}
+\frac{\partial u_k}{\partial x_i}\frac{\partial u_k}{\partial x_k}
+{\partial u_m\over\partial x_i}\frac{\partial u_k}{\partial x_m}
-\frac{\partial u_i}{\partial x_k}\frac{\partial u_i}{\partial x_i}
-\frac{\partial u_k}{\partial x_k}\frac{\partial u_i}{\partial x_k}
-{\partial u_m\over\partial x_k}\frac{\partial u_i}{\partial x_m}
\\
&=\frac{\partial u_i}{\partial x_i}
\Big(\frac{\partial u_k}{\partial x_i}
-\frac{\partial u_i}{\partial x_k}\Big)
+\frac{\partial u_k}{\partial x_k}
\Big(\frac{\partial u_k}{\partial x_i}
-\frac{\partial u_i}{\partial x_k}\Big)
-{\partial u_m\over\partial x_k}\frac{\partial u_i}{\partial x_m}
+{\partial u_m\over\partial x_i}
\frac{\partial u_k}{\partial x_m}.
\end{aligned}\label{e12}
\end{equation}
In \eqref{e12} we have replaced the sum $\sum_{j=1}^3$ by
$\sum _{j\in \{i,k,m\}}$ which is possible since $(i,k,m)$ is a
permutation of $(1,2,3)$.
Due to \eqref{e2} the right-hand side of \eqref{e12} can be
further transformed into
\begin{align*}
&-{\partial u_m\over\partial x_m}
\Big( \frac{\partial u_k}{\partial x_i}
-\frac{\partial u_i}{\partial x_k}\Big)
-{\partial u_m\over\partial x_k}\frac{\partial u_i}{\partial x_m}
+{\partial u_m\over\partial x_i}\frac{\partial u_k}{\partial x_m}
\\
&={\partial u_m\over\partial x_m}h_{i,k}
-{\partial u_m\over\partial x_k}\frac{\partial u_i}{\partial x_m}
+{\partial u_m\over\partial x_i}\frac{\partial u_k}{\partial x_m}.
\end{align*}
The proof is complete.
\end{proof}

\begin{lemma} \label{lem3}
 Let $u(x,t)$ and $p(x,t)$ be
$C^{\infty}\left(\mathbb{R}^3\times [0,\infty)\right)$
functions satisfying
\eqref{e1} and \eqref{e2} with $f_i(x,t)$ being identically
zero for $1\le i\le 3$,
and let $(i,k,m)$ be a permutation of $(1,2,3)$.  The following
relations hold for all $x\in \mathbb{R}^3$ and $t\ge 0$:
$$
\frac{\partial h_{i,k}}{\partial x_k}
+\frac{\partial h_{i,m}}{\partial x_m}
=\Delta u_i
$$
where $h_{i,k}$ is defined by \eqref{e9} and $\Delta u_i$
is the Laplacian of $u_i$ in the space variables.
\end{lemma}

\begin{proof}
 From \eqref{e9} we have
\begin{equation}
\frac{\partial h_{i,k}}{\partial x_k}
+\frac{\partial h_{i,m}}{ \partial x_m}
=\frac{\partial^2 u_i}{ \partial {x_k}^2}
-{\partial^2 u_k\over \partial x_i\partial x_k}
+\frac{\partial^2 u_i}{ \partial {x_m}^2}
-\frac{\partial^2 u_m}{ \partial x_i\partial x_m}.\label{e13}
\end{equation}
The right-hand side of \eqref{e13} can be rewritten, by \eqref{e2},
as
\begin{align*}
&\frac{\partial^2 u_i}{ \partial {x_k}^2}
+\frac{\partial^2 u_i}{ \partial {x_m}^2}
-\frac{\partial p}{\partial\partial x_i}
\Big( \frac{\partial u_k}{\partial x_k}
+{\partial u_m\over\partial x_m}\Big) \\
&=\frac{\partial^2 u_i}{ \partial {x_k}^2}
+\frac{\partial^2 u_i}{ \partial {x_m}^2}
+\frac{\partial p}{\partial\partial x_i}
\frac{\partial u_i}{\partial x_i}=\Delta u_i.
\end{align*}
This completes the proof.
\end{proof}

As an example of \eqref{e10} let us find another solution
to \eqref{e1}-\eqref{e3}.

\begin{lemma} \label{lem4}
 Let $b_j,\alpha_j\in \mathbb{R}$ satisfy
$\sum_{j=1}^3\frac{b_j}{\alpha_j}=0$.
Let $f_i(x,t)$ be chosen identically zero for $1\le i \le3$.
Then
$$
u_i^0=b_i\sin \Big(\sum_{s=1}^3{x_s\over \alpha_s} \Big)  , \quad
1\le i\le 3,
$$
is a smooth, periodic, bounded
and divergence-free initial vector field.
Let $g:\mathbb{R}\to \mathbb{R}$ be a smooth function with
$g(0)=g'(0)=0$.
The following family of functions $u$ and $p$ satisfy
\eqref{e1}-\eqref{e3}:
\begin{equation} \label{e14}
\begin{gathered}
u_i(x,t)=b_i e^{\beta t}\sin \Big( \sum_{s=1}^3{x_s\over \alpha_s}
+g(t) \Big)-g_0(t),\\
p(x,t)=g'_0(t)\sum_{j=1}^3x_j,
\end{gathered}
\end{equation}
where $g'(t)=g_0(t)\sum_{j=1}^3\frac{1}{\alpha_j}$ and
$\beta =-  v \sum_{j=1}^3 {1\over \alpha_j^2}$.
\end{lemma}

\begin{proof}
Let us write $z=\sum_{s=1}^3{x_s\over \alpha_s}+g(t)$ for brevity.
Then
\begin{gather*}
\frac{\partial u_i}{\partial t}
=\beta b_i e^{\beta t}\sin(z)-g'_0(t) + g'(t) b_i e^{\beta t}\cos(z),\\
- v \Delta u_i=  v  b_i e^{\beta t}
\sum_{j=1}^3{1\over \alpha_j^2}\sin (z),\\
\frac{\partial p}{\partial x_i}=g'_0(t),\\
\begin{aligned}
\sum_{j=1}^3u_j\frac{\partial u_i}{\partial x_j}
&=b_i e^{2\beta t}  \sin(z)\cos(z) \sum_{j=1}^3\frac{b_j}{\alpha_j}
-b_i g_0(t) e^{\beta t}\cos(z)\sum_{j=1}^3\frac{1}{\alpha_j}\\
&=-b_i g_0(t) e^{\beta t}\cos(z)\sum_{j=1}^3\frac{1}{\alpha_j}
\end{aligned}
\end{gather*}
since $\sum_{j=1}^3\frac{b_j}{\alpha_j}=0$.
Inserting the parts to \eqref{e1} shows that
\begin{align*}
&\frac{\partial u_i}{\partial t}
+\sum_{j=1}^3 {u_j\frac{\partial u_i}{\partial x_j}}- v \Delta u_i
+\frac{\partial p}{\partial x_i}\\
&=\Big( \beta + v \sum_{j=1}^3 {1\over \alpha_j^2}\Big)
b_i e^{\beta t}\sin (z)
+\Big(g'(t) -g_0(t)\sum_{j=1}^3 \frac{1}{\alpha_j}\Big)
 b_ie^{\beta t}\cos (z)=0
\end{align*}
by the conditions on $g'(t)$ and $\beta$ in Lemma \ref{lem4}.
\end{proof}

The simple reasoning leading to the solutions in Lemmas
\ref{lem1} and \ref{lem4}
is as follows. Looking at \eqref{e10} it seems that the
leading terms of
\begin{equation}
\sum_{j=1}^3 u_j\frac{\partial h_{i,k}}{\partial x_j}
-\frac{\partial u_m}{\partial x_m}h_{i,k}
+\frac{\partial u_m}{\partial x_k}\frac{\partial u_i}{\partial x_m}
-\frac{\partial u_m}{\partial x_i}\frac{\partial u_k}{\partial x_m}
=g(x,t)\label{e15}
\end{equation}
should cancel and leave a reminder $g(x,t)$ that can be obtained
from the time derivative of $h_{i,k}$. Then there is a first order
differential equation
$$
\frac{\partial h_{i,k}}{\partial t}- v \Delta h_{i,k}+g(x,t)=0
$$
which suggests that the solution is exponential
and in order to get periodic
initial values, trigonometric functions were selected.

In Lemma \ref{lem4} we first select
$u_i=b_i f\big( {\sum_{s=1}^3 {x_s\over \alpha_s}} \big)$
where $f$ is a smooth function to be determined.
This choice automatically gives
$$
\sum_{j=1}^3 u_j\frac{\partial h_{i,k}}{\partial x_j}=0
$$
because expanding it shows that
it has the multiplicative term $\sum_{s=1}^3{b_s\over \alpha_s}$ which
is zero by divergence-free condition \eqref{e2}.
The terms
\begin{align*}
&-\frac{\partial u_m}{\partial x_m}h_{i,k}
+\frac{\partial u_m}{\partial x_k}\frac{\partial u_i}{\partial x_m}
-\frac{\partial u_m}{\partial x_i}\frac{\partial u_k}{\partial x_m}\\
&=\Big(-\frac{\partial u_m}{\partial x_m}
 \frac{\partial u_i}{\partial x_k}
+\frac{\partial u_m}{\partial x_k}
 \frac{\partial u_i}{\partial x_m}\Big)
+\Big( \frac{\partial u_m}{\partial x_m}
\frac{\partial u_k}{\partial x_i}
-\frac{\partial u_m}{\partial x_i}
\frac{\partial u_k}{\partial x_m}\Big)
\end{align*}
also cancel automatically for the chosen function $u_i$.
Another way to cancel the terms is used in Lemma \ref{lem1}.
In Lemma \ref{lem1} we originally set $u_m=h_{i,k}$ by which
$$
u_m\frac{\partial h_{i,k}}{\partial x_m}
-\frac{\partial u_m}{\partial x_m}h_{i,k}=0.
$$
The remaining terms in the left side of \eqref{e15} are
\begin{gather*}
\frac{\partial u_m}{\partial x_i}\Big( u_i-
\frac{\partial u_k}{\partial x_m}\Big), \\
\frac{\partial u_m}{\partial x_k}\Big( u_k+
\frac{\partial u_i}{\partial x_m}\Big).
\end{gather*}
The divergence-free condition \eqref{e2} assuming
$u_m=h_{i,k}$ yields
$$
0=\sum_{j=1}^3\frac{\partial u_j}{ \partial x_j}
=\frac{\partial p}{\partial\partial x_i}\Big( u_i-
\frac{\partial u_k}{\partial x_m}\Big)
+\frac{\partial}{\partial x_k}\Big( u_k+
\frac{\partial u_i}{\partial x_m}\Big).
$$
The form \eqref{e4} is constructed such that it is divergence-free
and the term $g(x,t)$ in \eqref{e15} can be obtained from \eqref{e15}.
The way to obtain it is adding a function of $t$ to
$x_i$, $1\le i\le 3$.
The basic solution can be further modified by a function $g(t)$ as in
\eqref{e4} and \eqref{e14}.

Lemmas \ref{lem5} and \ref{lem6}, below, generalize Lemma \ref{lem4}.
Lemma \ref{lem5} shows that any periodic smooth function
$u^0_i=b_j g(\sum_j {x_j\over \alpha_j})$
with $\sum_j \frac{b_j}{\alpha_j}=0$
can be continued to smooth $u$ for zero external force since it
can be expressed by its Fourier series.
Lemma \ref{lem6} generalizes the solution to a case where there are two
different functions.

\begin{lemma} \label{lem5}
 Let $f_i(x,t)=0$ and
$$
u^0_i(x)=b_i\sum_{n=1}^{\infty} \Big(c_n\sin\Big(\sum_{j=1}^3
{n x_j\over \alpha_j}\Big)
+d_n\cos\Big(\sum_{j=1}^3 {n x_j\over \alpha_j}\Big)\Big)
$$
where $\sum_{j=1}^3 \frac{b_j}{\alpha_j}=0$. The following functions
solve \eqref{e1}-\eqref{e3}:
\begin{gather*}
u_i(x,t)=b_i\sum_{n=1}^{\infty} e^{\beta n^2 t}\Big(
c_n\sin\Big(\sum_{j=1}^3 {n x_j\over \alpha_j}\Big)
+d_n\cos\Big(\sum_{j=1}^3 {n x_j\over \alpha_j}\Big)\Big),\\
p(x,t)=0,
\end{gather*}
where $\beta=- v \sum_{j=1}^3 \alpha^{-2}_j$.
\end{lemma}

\begin{proof}
Direct calculations show that
\[
\sum_{j=1}^3 u_j\frac{\partial u_i}{\partial x_j}=0,\quad
\frac{\partial u_i}{\partial t}- v \Delta u_i=0.
\]
which completes the proof.
\end{proof}

\begin{lemma} \label{lem6}
 Let $f_i(x,t)=0$ and
$$
u^0_i(x)=b_{i,1}\sin\Big( \sum_{s=1}^3 {x_s\over \alpha_{s,1}} \Big)
+b_{i,2}\sin\Big(\sum_{s=1}^3 {x_s\over \alpha_{s,2}}\Big)
$$
where $\sum_j { b_{j,n}\over \alpha_{j,n}  }=0$ for $n=1,2$ 
and
$$
b_{i,1}\sum_{j=1}^3{b_{j,2}\over \alpha_{j,1}}={1\over \alpha_{i,2}} ,\quad
b_{i,2}\sum_{j=1}^3{b_{j,1}\over \alpha_{j,2}}={1\over \alpha_{i,1}}.
$$
Let $\beta_n=- v \sum_j \alpha_{j,n}^{-2}$ for $n=1,2$.
The following two functions solve \eqref{e1}-\eqref{e3}:
\begin{gather*}
u_i(x,t)=b_{i,1}e^{\beta_1 t}\sin
\Big(\sum_{s=1}^3 {x_s\over \alpha_{s,1}}\Big)
+b_{i,2}e^{\beta_1 t}\sin\Big(\sum_{s=1}^3 {x_s\over \alpha_{s,2}}\Big),\\
p(x,t)=\cos\Big(\sum_{j=1}^3 {x_j\over \alpha_{j,1}}\Big)
\cos\Big(\sum_{j=1}^3 {x_j\over \alpha_{j,2}}\Big)
e^{(\beta_1+\beta_2) t}.
\end{gather*}
\end{lemma}

\begin{proof}
Computing  $\sum_j u_j\frac{\partial u_i}{\partial x_j}$
shows that the term equals
$-\frac{\partial p(x,t)}{\partial x_i}$.
We mention that the conditions in the lemma imply
$\sum_j {1\over \alpha_{j,1}\alpha_{j,2}}=0,$
$\sum_j {1\over \alpha_{j,1}^2}=\sum_j {1\over \alpha_{j,2}^2}.$
\end{proof}

\section{Approaches to general initial values}

Let us first notice that the transform by the function $g(t)$
in Lemma \ref{lem1} and Lemma \ref{lem4} is not a coordinate transform of $(x,t)$ to
$(x',t')$ where $x'_j=x_j+g(t)$, $j=1,2,3$, and $t'=t$.
The equation \eqref{e1}
is not invariant in this transform and if
$$
\frac{\partial u_i(x,t)}{ \partial t}+\sum_{j=1}^3u_j(x,t)
\frac{\partial u_i(x,t)}{ \partial x_j}- v \Delta u_i(x,t)
+{\partial p(x,t)\over \partial x_i}-f_i(x,t)=0
$$
then
\begin{align*}
&{\partial u_i(x',t')\over \partial t'}+\sum_{j=1}^3u_j(x',t')
{\partial u_i(x',t')\over \partial x'_j}- v \Delta' u_i(x',t')
+{\partial p(x',t')\over \partial x'_i}\\
&-f_i(x',t') +g'(t)\sum_{j=1}^3{\partial u_i(x',t')\over \partial x'_i}=0.
\end{align*}

The transform that is used in Lemma \ref{lem1} and Lemma \ref{lem4},
\begin{gather*}
u(x,t)\to u(x',t')-g'(t),\\
p(x,t)\to p(x',t')+g''(t)\sum_{j=1}^3 x_j
\end{gather*}
keeps the initial values $u^0_i(x)$ and $f(x,t)$
fixed if $g(0)=g'(0)=0$.
It is a transform that can be done to any solution
of \eqref{e1}-\eqref{e3} but it  works only for  equation \eqref{e1}.
It is certainly not a generally valid coordinate transform.
Such should work with any equation.
It is not valid to think of the transform as
\begin{gather*}
u(x,t)\to u(x',t')-g'(t),\\
p(x,t)\to p(x',t'),\\
f_i(x,t)\to f_i(x',t')-g''(t)
\end{gather*}
because this changes the previously selected force $f_i(x,t)$.
In fact, what is done in Lemma \ref{lem1} is not a change of the
coordinate system. The force is kept at
the selected value at zero. The coordinates are kept at $(x,t)$ as they
originally are. The pressure is eliminated from \eqref{e1}
and \eqref{e2} as in \eqref{e10}.
The equation \eqref{e10} has several solutions for $u(x,t)$
and we find a family of solutions for the initial values
of Lemma \ref{lem1} and some solutions
cannot be smoothly continued to the whole space-time.
In Theorem \ref{thm3} we notice that it is possible to select a force
that picks up any of these solutions.

The equations \eqref{e1}-\eqref{e3} can be solved in a Taylor
series form, though summing
the Taylor series can be difficult. We write
$$
u_j=\sum_{n=0}^{\infty}\psi_{n,j}(x)t^n, \quad
p=\sum_{n=0}^{\infty} p_n(x)t^n, \quad
f_j=\sum_{n=0}^{\infty} f_{n,j}(x)t^n.
$$
Equation \eqref{e1} yields
$$
(n+1)\psi_{n+1,i}+\sum_{m=0}^n\psi_{m,j}\\
{\partial \psi_{n-m,i}\over \partial x_j}- v
\Delta \psi_{n,i}+p_n-f_{n,i}=0.
$$
These equations can be solved recursively by dividing
$$
v_{n,i}=\sum_{m=0}^n\psi_{m,j}
{\partial \psi_{n-m,i}\over \partial x_j}-f_{n,i}
$$
into two parts
$$
v_{n,i}=v_{n,i,1}+v_{n,i,2},
$$
where $v_{n,i,1}$ is divergency-free and $v_{n,i,2}$ has no
 turbulence; i.e.,
it can be obtained from some function $g(x,t)$ as
$$
v_{n,i,2}=\frac{\partial g}{\partial x_i}.
$$
Thus, what needs to be solved is a system
\begin{gather*}
(n+1)\psi_{n+1,i}- v \Delta \psi_{n,i}+v_{n,i,1}=0,\\
v_{n,i,2}={\partial p_n\over \partial x_i}.
\end{gather*}
We can see the non-uniqueness of the solution.
The division of $v_{n,i}$ into
the two parts is not unique: if $\Delta g=0$, then
$\frac{\partial g}{\partial x_i}$ can be inserted
to either $v_{n,i,1}$ or
to $v_{n,i,2}$. We selected a linear symmetric $g$
in Lemma \ref{lem1} in order to
have a nice periodic $u_i$.

The following approach is another way of finding $u(x,t)$ in
\eqref{e1} for a general initial vector field $u^0(x,t)$. In some
cases the method may yield closed form results easier than the
Taylor series method. If \eqref{e1}, \eqref{e2} and \eqref{e3} are
satisfied, $p$ can be derived by integration. Let us assume
$u(x,t)$ exists. We separate a multiplicative part $Y(t)$ such
that $u(x,t)=Y(t)X(x,t)$ where the scaling is $Y(0)=1$. If there
is no nontrivial multiplicative factor $Y(t)$ then let us set
$Y(t)\equiv 1$. If $(u^0)^{-1}$ exists (e.g. locally for a local
solution), and $g(x,t)=(u^0)^{-1}(X(x,t))$ is smooth then $g(x,t)$
can be expanded as a power series $\sum_{s=0}^{\infty}t^sg_s(x)$
of $t$ and the series converges in some small neighborhood of the
origin. Since $u(x,0)=u^0(x)$ we have $g_0(x)=x$. Let us write
$x'(x,t)=(x'_i(x,t))_{1\le i\le 3}$ where
$$
x'_i(x,t)=x_i+\sum_{s=1}^{\infty}t^sg_{s,i}(x).
$$ 
Thus
$u_i(x,t)=Y(t)u^0_i(x')$. Let
$$
h_{i,k}^0={\partial u_i^0\over \partial x_k}-
{\partial u_k^0\over \partial x_i}
$$
and
\begin{align*}
f_{0,i,k}(x,t)
&={\partial h_{i,k}^0\over \partial t}
+Y(t)\sum_{j=1}^3 u_j^0{\partial h_{i,k}^0\over\partial x_j}
- v \Delta h_{i,k}^0\\
&\quad -Y(t)\Big({\partial u_m^0\over \partial x_m}h_{i,k}^0
-{\partial u_m^0\over \partial x_k}{\partial u_i^0\over \partial x_m}
+{\partial u_m^0\over \partial x_i}{\partial u_k^0\over \partial x_m}
\Big).
\end{align*}
Clearly
\begin{align*}
\frac{\partial u_i(x,t)}{ \partial x_j}
&=Y(t)\sum_{r=1}^3{\partial x'_r(x,t)\over \partial x_j}
{\partial u_i^0(x')\over \partial x'_r}\\
&=Y(t){\partial u_i^0(x')\over \partial x'_j}+
Y(t)\sum_{s=1}^{\infty}t^s
\sum_{r=1}^3{\partial g_{s,r}(x)\over \partial x_j}
{\partial u_i^0(x')\over \partial x'_r}
\end{align*}
and thus
$$
h_{i,k}(x)=Y(t)h_{i,k}^0(x')+
Y(t)\sum_{s=1}^{\infty}t^s\sum_{r=1}^3\Big(
{\partial g_{s,r}(x)\over \partial x_k}
{\partial u_i^0(x')\over \partial x'_r}
-{\partial g_{s,r}(x)\over \partial x_i}
{\partial u_k^0(x')\over \partial x'_r}\Big).
$$
The interesting term is (here $Y'=dY/dt$)
\begin{align*}
{\partial h_{i,k}(x,t)\over \partial t}
&=Y(t){\partial h_{i,k}^0(x)\over \partial t}+
Y'(t)h_{i,k}^0(x')\\
&\quad +Y'(t)\sum_{s=1}^{\infty}t^s\sum_{r=1}^3
\Big({\partial g_{s,r}(x)\over \partial x_k}
{\partial u_i^0(x')\over \partial x'_r}
-{\partial g_{s,r}(x)\over \partial x_i}
{\partial u_k^0(x')\over \partial x'_r}\Big)\\
&\quad +Y(t)\sum_{s=1}^{\infty}st^{s-1}\sum_{r=1}^3
\Big({\partial g_{s,r}(x)\over \partial x_k}
{\partial u_i^0(x')\over \partial x'_r}
-{\partial g_{s,r}(x)\over \partial x_i}
{\partial u_k^0(x')\over \partial x'_r}\Big)
\end{align*}
because it lowers powers of $t$ and allows recursion.
Computing the terms in \eqref{e10} shows that for some functions
$Q_{s,i,k}(x,t)$, it holds
\begin{align*}
0&=\frac{\partial h_{i,k}}{\partial t}
+\sum_{j=1}^3 u_j\frac{\partial h_{i,k}}{\partial x_j}
- v \Delta h_{i,k}
-\frac{\partial u_m}{\partial x_m}h_{i,k}
+\frac{\partial u_m}{\partial x_k}\frac{\partial u_i}{\partial x_m}
-\frac{\partial u_m}{\partial x_i}\frac{\partial u_k}{\partial x_m}\\
&=Y(t)\sum_{s=1}^{\infty}st^{s-1}\sum_{r=1}^3
\Big({\partial g_{s,r}(x)\over \partial x_k}
{\partial u_i^0(x')\over \partial x'_r}
-{\partial g_{s,r}(x)\over \partial x_i}
{\partial u_k^0(x')\over \partial x'_r}\Big)
+\sum_{s=0}^{\infty}t^s Q_{s,i,k}(x,t)
\end{align*}
where $Q_{0,i,k}(x,t)=Y'(t)h_{i,k}^0(x')+Y(t)f_{0,i,k}^0(x')$.
Comparing the coefficients of the powers of $t$ individually,
and then inserting $t=0$ to the equation of each coefficient yields,
for $s\ge 1$,
\begin{equation}
\sum_{r=1}^3\Big({\partial g_{s,r}(x)\over \partial x_k}
{\partial u_i^0(x)\over \partial x_r}
-{\partial g_{s,r}(x)\over \partial x_i}
{\partial u_k^0(x)\over \partial x_r}\Big)
=-{1\over s}Q_{s-1,i,k}(x,0).\label{e16}
\end{equation}
 From \eqref{e2} we obtain
\begin{align*}
0&=\sum_{i=1}^3\frac{\partial u_i}{\partial x_i}\\
&=Y(t)\sum_{i=1}^3\sum_{r=1}^3 {\partial x'_r(x,t)\over \partial x_i}
{\partial u_i^0(x')\over \partial x'_r}\\
&=Y(t)\sum_{i=1}^3{\partial u_i^0(x')\over \partial x'_i}
+Y(t)\sum_{s=1}^{\infty}t^s\sum_{i=1}^3\sum_{r=1}^3
{\partial g_{s,r}(x)\over \partial x_i}
{\partial u_i^0(x')\over \partial x'_r}
\end{align*}
The first term in the right-hand
side vanishes because $u^0$ is divergence-free.
Again, comparing the coefficients of the powers of $t$ individually,
and then inserting $t=0$ to the equation of each coefficient yields,
for $s\ge 1$,
\begin{equation}
\sum_{i=1}^3\sum_{r=1}^3 {\partial g_{s,r}(x)\over \partial x_i}
{\partial u_i^0(x)\over \partial x_r}=0.\label{e17}
\end{equation}
It seems that solving \eqref{e16} and \eqref{e17} for $s=1$
gives $g_1(x)$, then we can derive
for $s=2$ and obtain $g_2(x)$ and
so on, and that function $Q_{s,i,k}$ contains only terms
$g_{s',j}$, $s'\le s$. However, \eqref{e16} and \eqref{e17}
do not necessarily determine even
$g_1$ and this approach must be modified.
This may be a direction for research how to obtain linear systems,
like \eqref{e16} and \eqref{e17}, from which to continue,
but we will not study this method more in this short paper.
Notice that this approach cannot show that a solution $u(x,t)$ exists.

\section{Discussion}

Theorem \ref{thm1} proves that the solutions to the 3-dimensional Navier-Stokes
equations for incompressible fluid are not always unique for the
initial values and for the periodic solutions discussed
in Statement D of \cite{feff}.
This is not surprising, as the proof of uniqueness requires
periodicity of $p(x,t)$.
Periodicity is not required in Theorem \ref{thm1} or in \cite{feff}.
Another (different) counterexample to uniqueness of \eqref{e1}
is given in
\cite{lady}. The proof of Statement D in \cite{feff} only requires Lemma 2.1,
Theorem 2.3 and Theorem 2.4. Let us next discuss the Clay problem and
whether Theorem 2.4 actually proves Statement D.

The EJDE article [4] states the following:
In \cite{feff} it is stated that we know for a long time
that the initial data $u^0(x)$ can be continued uniquely to some finite
time. 

Looking at \cite{feff} 
we see that there is nowhere any claim of uniqueness.
Was there such a claim in the Spring 2008 or Summer 2009? 
I informed Fefferman and CMI about the need to correct the problem statement
but never received any answer. Did they modify the problem statement or not? 
This cannot be 
proven, but the following facts point out to the possibility that 
\cite{feff} was silently modified between Spring 2008 and Summer 2010.

I discussed [4] with an expert of the Clay Navier-Stokes problem, 
Claes Johnson from KTH in September 2012. Johnson has coauthored in 2008 
an article questioning if CMI Navier-Stokes problem is well-posed, so he 
should have read the official problem statement at that time. 
Johnson's first argument 
against [4] was that strong solutions for Navier-Stokes equations are unique 
and therefore [4] must be wrong. Johnson gave the Fefferman quote as a 
support of his words. This quote is in the official problem statement on 
page 2 of \cite{feff} : it is known that (A) and (B) 
hold (also for $\nu=0$) if the time interval $[0,\infty)$ is replaced by a
small time interval $[0,T)$. However, looking at (A) and (B), that is 
Statements A and B in \cite{feff} on page 2,
we can see that these statements do not make any claim
that solutions can be continued uniquely. Furthermore, when we search for
versions of the official problem statement in the web, we do not find
any version where (A) and (B) would ever have claimed that solutions are 
unique. This appears strange considering Johnson's argument.

There is also the following case: in Spring 2008 [4] was in review in the 
journal AASF. The referee of AASF gave the same reference to the Fefferman 
quote as a justification of their claim that strong solutions are unique. I 
looked up Statements A and B at that time, and they did claim that solutions 
are unique. At that time I wrote the sentence to [4] stating that \cite{feff}
makes a claim of uniqueness. Was the AASF referee, and I, maybe dreaming?

Finally there is also the following case: in 2009, while checking [4], Jouni
Luukkainen sent to CMI an email requesting that they correct the problem 
statement, the solutions are not unique unless $p$ is required to be periodic.
Luukkainen would not have sent this email if the problem statement did not 
make the claim of uniqueness. Was he maybe also dreaming?

Now the official problem statement says, or should we say, has always stated,
that either (A) and (B) hold, or there is a blowup solution. This apparently 
means that
either (A) holds or (C) holds, and either (B) holds or (D) holds. 

This claim is also wrong,
We can very well have the situation that (B) holds and (D) also holds. 
Statement B says that for initial data $u^0(x)$ and zero external force 
$f$ there exists a smooth solution 
that can be continued to the whole space-time. Statement D says that there 
exists initial data $u^0$ and a force $f$ such that all solutions blow up.
These statements are not exact opposites. There can be $u^0$ such that for
identically zero $f$ there exists a smooth solution in the whole space-time,
but there also exists a force $f$, that is not identically zero but has zero 
value everywhere, such that there is only one solution under this force and
that solution blows up. Such a force can be defined as a feedback control
force. If there are many solutions, there can be one solution that can be
continued to the whole space-time, and one blowup solution. We can select
a feedback control force that steers all solutions (for which this force
is defined) to join the blowup solution. This has been done in Theorem 2.4. 

Earlier, when the official problem statement claimed in (A) and (B) that
solutions are locally unique (if it ever claimed this, we cannot prove it) 
it was very clear that the solution in [4] is a solution of CMI 
Navier-Stokes problem, but even now when (A) and (B) make no claim of 
uniqueness, the solution of [4] is still a valid one. The official problem
statement should still be corrected as [4] demands by stating: Unless 
Theorem \ref{thm3} is accepted as a proof of Statement D
in \cite{feff}, the official problem statement for the millennium
problem must be corrected.

Statement D in \cite{feff} 
has in the title Breakdown of Navier-Stokes Solutions on 
$\mathbb{R}^3/\mathbb{Z}^3$ 
but the statement is given on $\mathbb{R}^3$.
It is the statement that has to be 
shown. It can be given any title. The statement lists all conditions (1), (2),
(3), (8), (9), (10), (11) in [2].
Statement D does not require the pressure to be space-periodic.

If you very much object to concluding that Statement D does not require 
space-periodicity for $p$ and think that of course $p$ must be space-periodic
since the title of Statement D says $\mathbb{R}^3/\mathbb{Z}^3$, and that
this counterexample only works because of this minimal formal error in the 
problem statement, please consider Statement C. 
Take a solution filling the conditions of Statement A for zero
external force and make the transform in Section 3 to that solution. 
Statement C does not set any growth conditions to the pressure. 
The external force has value zero, so there is no problem from it. The only 
additional condition you have to check is (7) in \cite{feff}, bounded energy.
In the transform of Section 3 your $u(x,t)$ has been transformed 
to $u(x',t)-g'(t)$ where 
$x'_j=x_j+g(t)$. The bounded energy condition in \cite{feff}
 requires that 
$$\int_\mathbb{R}^3 |u(x,t)|^2 dx$$ is bounded by some number.
Select $g(t)$ that is almost zero and has a very thin spike that goes 
to infinity. When $g(t)$ grows, $x'_j$ grows and $u(x',t)$ vanishes in 
infinity. It is easy to find a function $g(t)$ that yields bounded energy.
Statement C clearly says $\mathbb{R}^3$ in the title and we can create
in almost exactly the same way a counterexample proving Statement C. 

The solutions are required to be physically reasonable. 
There is the physicality condition (7) in \cite{feff} 
but it is not required in 
Statement D. In Statement D physical reasonability is (10) and (11) of 
\cite{feff},
which we already included in Statement D. The initial data is $u^0$. 
That is physical here and in Theorem 2.4 of [4]. 
Our time derivatives of $u$ are not physical at $t=0$, 
but they are not required to be physical.

The force in Theorem 2.4 of [4]
is a feedback control force of the type as the steering force steering a 
car moving on a straight road. If the car moves to the direction of the 
road, is stays on the road without steering. If it moves to other 
directions, it drops off the road. A steering force turns the car to the 
direction of the road and then keeps it there. When the car is in the 
direction of the road, the force has zero value. How much work the force 
has to do to get the car on the correct direction depends on what was the 
original position of the car. If it already was in the correct direction, 
the force has always zero value. This last case is the situation in 
Theorem 2.4. 

Mathematically we can consider a linear system $du/dt=ku+c+f$ where $k$ 
and $c$ can have many values and one solution for $f=0$ is $dv/dt=av+b$. 
A feedback force $f=du/dt-dv/dt$ gives a solution $f=0$, $u=v$ in case 
$k=a$ and $c=b$. So, if a linear system is started in correct direction, 
the feedback force always has 
the value zero. If the system has some other starting values $k,c$, the 
force $f$, which in this case is not the same as given above, first moves 
the system to the correct direction and then becomes zero. If the system 
is nonlinear, like the Navier-Stokes equation, 
we linearize it at a vicinity of $t=0$ and the control force behaves in 
the same way.  After some time the feedback control force sets $u=v$ and $f$ 
becomes zero. After that point $u=v$ for all time because strong solutions are 
unique after $u$ equals $v$ for an interval. 

What then is the initial position in Theorem 2.4? Does the external force 
$f$ first need to move the system into correct direction? 
No. Actually $f$ cannot move the system into the correct direction. It must
do a selection as in Theorem 2.4. Let me explain this issue. 

Statement D has upper limits in (4) of \cite{feff} 
(equation (2.7) in Theorem 2.4) to the absolute value of all 
derivatives of $f$, and moving a solution with some $g(t)$ to have 
the $g(t)$ of the blowup solution must require force. As there are 
arbitrarily large functions $g(t)$, the force must in some cases 
be higher than any limit. Therefore we see that a single feedback control
function $f$ cannot steer all solutions to the blowup solution. If we want
to use a feedback control function, it can only be defined for a subset of 
the values of $g(t)$, such functions $g(t)$ that are sufficiently close to
each other that they can be steered under the constraints (4)
of \cite{feff} poses. 

This situation is more generally true for feedback control forces. Usually
the same feedback control force cannot be used to control all possible 
moves of the system. In the car example, the control system includes the 
turning front wheels and the steering wheel, among othe things. It is not 
possible to steer the car so that it moves sideways. If you want to do that,
you have to use another feedback control force, for instance a bulldozer. 
We must define the feedback control force so that it can control the values of
$g(t)$ that are in the subset of values of $g(t)$ that the force controls.
Thus, when a feedback control force is selected, we also select the subset
of values of $g(t)$ that it controls.

This phenomenon is here called selection and its opposite is called filtering. 
Feedback control forces have two ways of working: they can steer a 
solution to follow the chosen solution, which requires work, or they can
filter solutions. Filtering does not require work: when choosing a feedback
control function we can only get a selected solution, a solution for which 
this particular control function is used for, not a filtered one. They are
impossible for the control function.
 
Selection in our case can be explained in the following way: 
Physical fluid has certain time derivatives for $u$ at $t=0$. These time
derivatives are determined by the function $g(t)$.
The fluid does not have all possible time
derivatives at the same time, but different values of (analytic) 
$g(t)$ correspond to different instances. 
Each feedback control force $f$ can control only a 
subset of all possible values of $g(t)$, so when we select $f$ we also
select the subset of $g(t)$ which is controlled by this $f$. 

We define a feedback control force that can control only one value
$g(t)$. The force $f$ has only one $g(t)$ in the subset. Thus, selecting
this $f$ implies selecting the particular $g(t)$. We have changed the selection
of the force $f$ to a selection of the function $g(t)$. This is what Theorem 
2.4 does. 

As an example, consider
a situation where we have a large set of different balls, they correspond
to the functions $g(t)$, and the task is to select what to do, this 
corresponds to selecting the force $f$. If we select an activity that has a 
purpose, such as playing tennis, then we also at the same time select a ball
that can be used in that activity. If we select playing tennis, we do not have
to ask how can we play tennis with a football. A feedback control force has
the purpose of controlling. Therefore selecting such a force selects those 
functions $g(t)$ that the force is meant to control. We created a force that
only controls one particular $g(t)$. Selecting that force also selects that 
$g(t)$. If the problem statement is not supposed to allow feedback control
forces, or more generally forces that have a purpose, it must be clearly 
stated in the problem description. Nothing in \cite{feff} 
forbids us from selecting 
such a feedback control force.

We see that by selecting the force $f$, we select the only one solution 
that has the time derivatives of $u$ at $t=0$ so that the initial position 
is that of $v$, the blowup solution. As there is only one solution, all 
solutions for this force have the blowup. 
Then $f$ has value zero from the beginning $t=0$ and $u=v$ 
all the time as in Theorem 2.4. 

Is a feedback control force allowed in Statement D? 
The official problem statement does not deny such a force. 
Therefore it is allowed. The problem statement only says: 
$f$ is a given, externally applied force (e.g. gravity). 
Given means that the term $f$ is not an unknown like $u$ and $p$, 
but given. In Statement D it can be selected. A given force can very 
well be given either explicitly or implicitly as a function of $u$. 
External forces in NSE often have feedback from $u$. 
Gravity, the only example mentioned, has feedback from the distribution 
of the fluid mass. Fluid has mass or gravity had no effect. Mass 
creates gravity. We do not know if the whole space is filled with 
fluid. Thus, the mass distribution is a function of $u$. 
The gravitation force has an additive component given by a function of $u$. 
In general, the interest in fluid turbulence is largely caused by research 
on propulsion systems. Propulsion systems are a part of a steering system 
and thus, they have feedback from water motion. As the CMI problems are 
intended for a wide public and they explain even simple things to the 
reader, they certainly are intended to contain all assumptions. 
As there are no explicit limitations to the external force, there are 
no implicit assumptions that the external force must fill. Thus, $f$ 
can be a feedback control force. 

This is the whole solution in [4]. Did we see any misinterpretation 
of the problem description? No. 
A misinterpretation
is something that violates some statement in the problem statement. 
Everything we did was in accordance with the official problem statement 
\cite{feff}.

\section{Answering critizism}

The article [4] has been so far discussed with 21 mathematicians or physicists.
Nobody has presented a valid argument against the proof of Statement D in [4]. I will go through the critizism against [4].

Grigory Rozenblioum noticed a small error in [4]. If does not affect the proof,
and has been corrected here by the first footnote (just before Lemma 2.1). 
The error was that [4] and \cite{feff} do not
require the solutions to be analytic, only smooth. If $g(t)$ can be 
nonanalytic, then we can for instance define
$$
g(t)={1\over a-t}e^{-{1\over x^2}}, \, g(0)=0.
$$
This smooth function has all derivatives zero at $t=0$ and a singularity
at $t=a$. Clearly, $u(x,t)$ which has this $g(t)$ has at $t=0$
the same time derivatives as $u(x,t)$ which has $g(t)=0$. Uniqueness of strong
solutions cannot be achieved by selecting the time derivatives of $u(x,t)$
at $t=0$ for nonanalytic solutions. The second claim of Rozenblioum was 
incorrect, the erronous sentence is not used in the proof. The second footnote
(in Theorem 2.4) explains this issue. It seems unnecessary to publish an 
errata for the EJDE
article since the only error Rozenblioum noticed is in an informal description
of Lemma 2.1 and if understood in the correct context it is correct.  
Rozenblioum gave a seminar on [4] but his objective was not to see if the 
proof is correct or not. Rozenblioum thought he had found a serious error from
[4], did not contact me about it, and probably told his students that there
is an error. Later, Fall 2012, Rozenblioum wrote to me that he stopped 
reading when noticing that [4] does not require $g(t)$ to be
analytic. When shown that this does not invalidate the proof, Rozenblioum
stated that he will not discuss whether [4] solves the Clay problem or not.
It does not interest him. Then why did he present [4] in a seminar?

The article [4] was in review in five journals. This was interpreted by
Rozenblioum as indicating that the article was controversial. Actually, as 
anybody who has tried to submit a manuscript with a solution to a well-known
difficult problem knows, 
it is very common that journals do not review such papers
correctly. Of these reviews, the editor of JDE first immediately claimed that 
there are errors, but when asked to state the errors changed the reason for 
rejection of the article to it not being interesting to the readers. 
Thus, JDE did
not review the article. AASF referred to the Fefferman quote of uniqueness
in \cite{feff}, which, as we can see, does not exist. 
Because of this nonexistent quote, AASF concluded that [4] must be wrong. 
AASF did not agree to review the paper again even though their referee was 
shown incorrect. 
JHDE rejected the article 
because the original proof of Theorem 2.4 was not for them satisfactory. The
proof was rewritten, but JHDE did not agree to review the paper another time.
The referee of JHDE also stated that it is well-known that strong solutions 
are not unique. Next the manuscript was submitted to JDDE, which
rejected the paper because it is well-known that strong solutions are 
unique, as is proven by a theorem in \cite{tom}. This proof was shown to use 
an implicit assumption of space-periodicity of $p$, but JDDE did not agree 
to review the paper again. 
At this point I had two exactly contradictory referee statements: one from JHDE
stating that strong solutions are well-known to be nonunique and from JDDE
that strong solutions are well-known to be unique. With these referee 
statements I managed to get a group of Finnish mathematicians and physicists
to look at the manuscript. Those who allowed being mentioned are thanked 
in the acknowledgements. Finally the manuscript was submitted to EJDE.
The editor of EJDE sent to me after six months a report showing that eight
referees had refused to review the article. Finally EJDE did find referees
who reviewed and accepted it.

There have also been some discussions with experts. 
Jouni Luukkainen from the small group of Finnish mathematicians stated that 
a given, externally applied force (e.g. gravity) cannot refer to a feedback 
force. He did not explain why he thinks so. Juha Honkonen noted that the force
is not energy preserving and Fung Lang made later a similar comment that energy
is not correctly obtained by integrating over one space period. This is caused
by the pressure not being space-periodic. Both accepted this answer. 
I tried to get an opinion from Terence Tao, who wrote a blog on Navier-Stokes 
equations and had a discussion forum where he answered to questions from 
people. I sent a post telling that [4] was published. Tao dropped all posts 
from me until I changed the user name. I did get one comment from Tao. He
claimed that the feedback force in Theorem 2.4 is a circular 
definition. It is not. It is a quite normal feedback control force 
definition. Tao did not want to discuss the article any more and forbid me
from writing any more to the discussion board on his web-page. 
Another mathematics blogger, Claes Johnson, at least did not drop my 
posts and actually did send to me two emails and initially agreed to 
discuss [4]. Johnson claimed that strong solutions are unique referring to 
the nonexistent Fefferman quite. When shown wrong, he did not want to discuss 
further. 
Internet 
characters cowgod2 and Robert Coulter (not the mathematician with this name)
discussed [4] for a long time in the web. 
Cowgod2 finally accepted the proof - first he argued that strong solutions 
are unique. Robert 
Coulter also accepted the proof but claimed that the problem solved is a 
misunderstaning of the CMI problem. He did not explain where the 
misunderstanding is but in the end tried to argue that feedback forces 
cannot be used. That they are never used in the Navier-Stokes equations. 
It just happened that Enrico Bombieri had sent me a list of papers where 
feedback forces are used in Navier-Stokes equations. His goal was to show 
that feedback control forces are nothing new in that field, so I made no 
discovery. The list did invalidate Coulter's argument. Coulter is true in 
that respect that it seems that several experts of 
the Navier-Stokes equations do not know feedback control forces. 
Especially they do not seem to know that a feedback control force can
be used to select a solution, i.e., to transfer the selection of the force 
to selection of $u(x.t)$. 
The CMI finally gave a reply after over two years from publication of [4]
had passes. The CMI stated that the claim that [4] solves the CMI problem
is based on a misinterpretation of the problem description. 
I think they broke the US law of skill contests by not setting up a group 
of experts to check [4]. Two years from publication in a journal of good reputation had passed and [4] has not been refuted or disputed
in any scientific forums. Therefore it is generally 
accepted by the mathematical community. That is what generally accepted means.
Because of these two conditions were filled, the CMI rules require a group 
to be set to check the solution. In skill contests stated rules must be 
followed. But who would like to sue those people. The CMI did not 
explain where they see a misinterpretation. In fact, there is no 
misinterpretation. Everything is as in \cite{feff}.
As Bombieri had already sent me the list, I asked him if he could tell where
I have a misinterpretation of the problem statement as the CMI refused to 
explain or to answer anything.
Bombieri wrote that he is no expert on the topic and cannot say if
[4] solves the Clay problem or not, and that it is a nonmathematical question 
to answer what in CMI's opinion is a solution to their problem. 
Was it not supposed to be a precise mathematical question? 

Other comments concerning [4] either accepted the proof or merely stated that
I will not get the prize. That was a good prediction, though an easy one to 
make. Still, it is not a mathematical argument. The article [4]
was sent to over 50 experts. Most did not answer. It has been in arxiv since 
2008 and published in a journal since 2010. 
It is correct to
state that there are no valid mathematical arguments against the proof of
Statement D in [4]. The Clay Navier-Stokes problem is solved, but we obtained
other interesting conjectures as well: the CMI may be prepared to break the 
skill contest law if needed, Clay official problem statements may be silently 
modified when a solution is in journal review, mathematics professors maybe 
arrange seminars on somebodys paper with the only intention of ridiculing the 
author and with no intention of trying to understand the paper, journals maybe
do not treat all manuscripts correctly and possibly the referees make errors
and refuse to re-evaluate the paper, experts maybe forbid some people to enter
their nominally open web discussion boards, 
experts possibly do not answer, some
experts may not understand the classical theorems of their own field. I cannot
prove any of these claims, but there is some indication that someting like
this may be true.

\section{Conclusions}

We conclude that Theorem 2.4 is indeed a valid proof of Statement D. 
Theorem 2.4 shows that there exists a space-periodic solution that cannot 
be continued to the whole space. The initial conditions and the external 
force in this solution confirm fully to the requirements of the problem 
statement. The solution is given explicitly. 

But the mathematicians do not like this solution.

Thus, the problem statement should be corrected. Let us make some 
observations concerning this issue. The counterexample in Theorem 2.4 
was possible because Statement D does not require the pressure to be 
space-periodic and because it does not forbid feedback control forces. 
Let us see what can be done to fix the problem. 

Forbidding feedback control forces does not help. Usually by a counterexample
to existence is meant that does there exist such $u(x)^0$ and such $f$ 
that the solution has a blowup. The answer is yes: the solution $U$ in
Theorem 2.4 has a blowup with the force $f=0$. However, as \cite{feff}
is now formulated, this argument may not be sufficient.
As \cite{feff} now is, one can forbid feedback control 
forces, but the issue has to be carefully considered, since gravity is a 
feedback force. Maybe we want to forbid forces that have a purpose, such 
as control forces? It was the purpose of the force that allowed us to link 
the selection of the force $f$ to the selection of $g(t)$. 
 
How about space-periodicity? It does not help to require space-periodicity 
for $p$ in Statement D. If we add a condition that 
the pressure be space-periodic in Statement D, the solution of [4] still 
exists. Only it does not fill the required conditions and is therefore a 
counterexample to existence. Proving Statement D requires finding a 
counterexample to existence. If we reformulate Statement D in a torus 
$T^3=\mathbb{R}^3/\mathbb{Z}^3$ instead of $\mathbb{R}^3$, 
Statement D is fixed, but a similar counterexample can be 
made for Statement C. There are no restrictions to pressure 
in Statement C and we really cannot put restrictions on pressure: 
pressure can be eliminated from the equation at least for $f=0$ (see [4], 
there it is done), and pressure is calculated from $u$. If we require $p$ to 
behave nicely, we assume $u$ to behave nicely. Then there is no sense to 
prove that $u$ behaves nicely. 

We could change the problem statement in such a way that the transform 
in Section 3 is not
allowed, but are we sure that it is the only transform? Where is a proof of 
that? It would be natural to forbid the transform in Section 3 since
the goal of the Clay problem statement is to have a 
physically reasonable problem on Navier-Stokes equations. Allowing this 
transform, and now it is allowed in \cite{feff}, 
already moves the solutions to 
the dreamworld. Physical fluid always has time derivatives defined at 
$t=0$. Such solutions as come from the transform are unphysical. We again
must say that the problem statement is very poor. 

The easiest change is to make Statement D the exact opposite of Statement B.
That is, to state that the external force is identically zero in Statement D, 
but this is not what we want. 

It is of course possible to state the official problem statement correctly 
in some
way but it is not obvious what way CMI will use. Therefore the problem 
statement must first be given a correct formulation before one can know what
exactly should be proven. [4] proves exactly what was stated in Statement D.
If that is not what should have been proven, then the problem statement is
incorrectly formulated.

\subsection*{Acknowledgments}
I want to thank the following people: Henryka Jormakka and
Jouni Luukkainen for reading most of the paper and suggesting valuable
improvements to the text, Juha Honkonen and Eero Hyry
for useful discussions concerning the beginning of the paper,
Matti Lehtinen for interesting discussions on many issues, and
the anonymous referees for their comments.
Without your help this paper would not have been checked.
I also thank Grigory Rozenbioum for pointing out that if the function $u(x,t)$
is not analytic, the solutions are not unique. This observation has been 
taken care of by a minor modification.

\end{document}